\documentclass{article}

\usepackage[T1]{fontenc}
\usepackage[english]{babel}

\usepackage[a4paper,top=2cm,bottom=2cm,left=2cm,right=2cm,marginparwidth=1.75cm]{geometry}

\usepackage{amsmath}
\usepackage{comment}
\usepackage{bbold}
\usepackage{dsfont}
\usepackage{graphicx}
\usepackage[colorlinks=true, allcolors=blue]{hyperref}
\usepackage{amsthm}
\usepackage{wrapfig}
\usepackage{amssymb}
\usepackage{stmaryrd}
\usepackage{csquotes}
\usepackage{biblatex}
\usepackage{orcidlink}
\newtheorem{prop}{Proposition}
\newtheorem{theorem}[prop]{Theorem}
\newtheorem{lemma}[prop]{Lemma}

\newtheorem*{remark}{Remark}

\newcommand{\Mod}{\,\mathrm{mod}\,}
\newcommand{\e}{\mathrm{e}}
\newcommand{\de}{\,\mathrm{d}}
\newcommand{\ii}{\mathrm{i}}
\newcommand{\Po}{\mathcal{P}}

\newcommand{\I}{\mathcal{I}}

\newcommand{\Z}{\mathds{Z}}

\newcommand{\R}{\mathds{R}}

\newcommand{\Ind}{\mathbb{1}}
\newcommand{\eps}{\varepsilon}

\newcommand{\starsum}{\ \sideset{}{^*}\sum}

\title{A Bombieri-Vinogradov theorem for exponential sums over products of $k$ primes}
\author{P\'ea Bazin \orcidlink{0009-0000-7184-9130} \footnote{Universit\'e Paris Cit\'e, Sorbonne Universit\'e, CNRS \\ Institut de Math\'ematiques de Jussieu-Paris Rive Gauche, 75013 Paris, France. \\ Email: bazin@imj-prg.fr}}
\date{}

\begin{document}
\maketitle

\begin{abstract}
    We prove a Bombieri-Vinogradov type theorem for exponential sums over products of $k$ primes. As an application, we show the lower bound $$\sup_{n\le x} \left|\sum_{m\le n} \Ind_{\Omega(m) = k} \e(\alpha m)\right| \gg x^{1/6 - \eps}$$ for $2\le k\le (2-\eps)\log\log x$ and $\alpha\in\R,$ where we noted $\e(\beta) := \e^{2\ii\pi\beta}.$ 
\end{abstract}

\section{Introduction}
\subsection{Background}

We note $\Omega(n)$ the number of prime factors of $n$ counted with multiplicity, $f_k(n) := \Ind_{\Omega(n) = k},$ and $$\Pi_k(x) := \sum_{n\le x} f_k(n) = \#\{1\le n\le x : \Omega(n) = k\}.$$ An asymptotic estimate of $\Pi_k$ when $k\le (2-\eps)\log\log x$ has been found by Sathe and Selberg \cite{sathe, selberg}, reminded below. Defining  
\begin{equation}\label{def-H}
    H(z) := \frac{1}{\Gamma(z+1)} \prod_p \left(1 - \frac{1}{p}\right)^z \left(1 - \frac{z}{p}\right)^{-1},
\end{equation}
we have
\begin{lemma}[Sathe \cite{sathe}, Selberg \cite{selberg}]\label{sathe-selberg}
    Let $\eps > 0.$ Then we have for $k \le (2-\eps)\log\log x,$ $$\Pi_k(x) = \frac{x}{\log x}\frac{(\log\log x)^{k-1}}{(k-1)!}\cdot \left( H\left(\frac{k-1}{\log\log x}\right) + O\left(\frac{1}{\log\log x}\right)\right).$$
\end{lemma}

We are interested here in the exponential sum over products of $k$ primes, $$\Pi_k(x;\alpha) := \sum_{n\le x} f_k(n) \e(\alpha n).$$
This exponential sum has been studied by Dupain, Hall and Tenenbaum \cite{dht, tenenbaum}, who obtained the upper bound \cite[Theorem A]{tenenbaum}
\begin{equation}
    \Pi_k(x;\alpha) \ll \Pi_k(x)\cdot \left(\frac{\log\log\log x}{\log\log x} + R(x,\alpha)^{-\eps/2}\right)
\end{equation}
for $2\le k\le (2-\eps)\log\log x,$ where $R(x,\alpha) := \min\left\{Q: \exists a,q\le Q, \left|\alpha - \frac{a}{q}\right| \le \frac{Q}{qx}\right\}.$
\begin{remark}
    The Sathe--Selberg and Dupain--Hall--Tenebaum results have also been proven when $f_k$ is replaced by $h_k := \Ind_{\omega(n) = k}$ the indicator function of integers with $k$ distinct prime factors. Moreover, for $h_k,$ these results also hold in the larger range $$k \le C \log\log x$$ (with constants depending on $C$). The proof of our results below can then also be easily adapted to obtain the same results for $h_k,$ in the larger range $k\le C\log\log x.$
\end{remark}

We obtain in this paper a Bombieri--Vinogradov estimate for $\Pi_k(x;a/q)$ (Theorem \ref{main} below) in the same range of $k,$ from which we deduce the lower bound $$\sup_{n\le x} |\Pi_k(n;\alpha)| \gg x^{1/6-\eps}$$ uniformly in $\alpha\in\R$ (Theorem \ref{uniform-pik} below) using the author's previous work in \cite{self}.

\subsection{Main results}

The main purpose of this paper is to prove the following Bombieri-Vinogradov estimate for exponential sums over products of $k$ primes,

\begin{theorem}\label{main}
    Let $\eps, C > 0.$ There exists $B$ depending only on $C$ such that when $Q \le x^{1/3}(\log x)^{-B}, \theta \le Q^{-3}(\log x)^{-B}$  and $1 \le k \le (2-\eps)\log\log x,$ we have
    \begin{equation}\label{main-eq}
        \sum_{q\le Q} \max_{(a,q) = 1} \max_{|\lambda|\le\theta} \left|\sum_{n\le x} f_k(n) \e\left(\left(\frac{a}{q}+\lambda\right)n\right) - \frac{1}{\phi(q)}\sum_{st = q} \mu(s) t\sum_{\substack{n\le x \\ t|n}} f_k(n)\e(\lambda n) \right| \ll \frac{x}{(\log x)^C}.
    \end{equation}
\end{theorem}

Theorem \ref{main} already appears in the literature in a paper of Yao \cite{yao}, however it appears their proof fails when $Q \ge x^{c_0},$ where
\begin{equation*}%\label{def-co}
    c_0 := \inf_{1/2 \le \sigma \le 4/5} \frac{1-\sigma}{6\frac{1-\sigma}{2-\sigma} - 1/2} = 2\left(2\sqrt{3}-1\right)^{-2} \approx 0.3294
\end{equation*}
(here, the inf is reached for $\sigma = 1 - \big(2\sqrt{3}-1\big)^{-1} \approx 0.5942$).
Indeed, the proof requires on \cite[p. 2008]{yao} the bound \begin{equation}\label{cond29}
        x^{\sigma + \eps} Q^{6\frac{1-\sigma}{2-\sigma} - 1/2} (1+\theta x)^{3\frac{1-\sigma}{2-\sigma} - 1/2} \ll \Pi_k(x)(\log x)^{-A}
\end{equation} for $1/2 \le \sigma \le 4/5,$ which only holds up to $Q = x^{c_0-\eps}$ when $\theta < 1/x.$ Our method is able to recover the missing range and reach the $x^{1/3-\eps}$ bound.

Simpler cases of Theorem \ref{main} have already been proven. Indeed, in the case of primes ($k = 1$), Theorem \ref{main} has been proven by Liu and Zhan \cite[Theorem 3]{liuzhan}. For general $k,$ Wolke and Zhan \cite{wolkezhan} proved the classic Bombieri-Vinogradov theorem for products of $k$ primes, that is
\begin{equation}\label{bv-pik}
    \sum_{q\le Q} \max_{(a,q) = 1} \left|\sum_{\substack{n\le x \\ n \equiv a\Mod q}} f_k(n) -\frac{1}{\phi(q/t)}\sum_{\substack{n\le x \\ (n,q) = 1}} f_k(n) \right| \ll \frac{\Pi_k(x)}{(\log x)^C}
\end{equation}
for $Q\le x^{1/2-\eps}.$

As an application, we prove in the last section a lower bound on exponential sums over products of $k$ primes using Theorem \ref{main} and results from the author's previous paper \cite{self}. 

\begin{theorem}\label{uniform-pik}
    Let $\alpha\in\R$ and $\eps, \eps' > 0.$
    Then for $k \le (2-\eps')\log\log x$ and $y\le x^{1/3 - \eps},$ we have  $$\frac{1}{x} \sum_{-y < n\le x} \left|\sum_{\substack{n< m\le n+y \\ 1 \le m \le x}} f_k(m) \e(\alpha m)\right|^2 \gg y \left(\frac{1 + |k-\log\log x|}{\log\log x}\right)^2 \left(\frac{\Pi_k(x)}{x}\right)^2,$$ where the implied constant depends only on $\eps, \eps'.$ In particular, uniformly in $\alpha\in\R,$ we have
    $$\sup_{n\le x} \big|\Pi_k(n;\alpha)\big| \gg x^{1/6 - \eps}.$$
\end{theorem}

\subsection{Notation}\label{notation}

The letter $\chi$ is reserved for Dirichlet characters. We use $\sum_{\chi\Mod q}$ to denote a sum over all characters modulo $q,$ and $\sum_{\chi\Mod q}^*$ for a sum restricted to primitive characters modulo $q.$

For a given arithmetic function $f,$ we note $F(x) := \sum_{n\le x} f(n).$ We also note for $\theta$ a function (usually a character $\chi$ or $\e_\lambda : n\mapsto \e(\lambda n)$), $$F(x; \theta) := \sum_{n\le x} f(n) \theta(n),\text{ and } F(x; q, 0) := \sum_{n\le x} f(n) \Ind_{q|n}.$$
When $f = f_k,$ we write $\Pi_k$ instead of $F.$

%We fix $\alpha\in\R$ and $f$ an arithmetic function, which will later be taken to be $\Lambda$ or $\Ind_{\Omega = k}.$ We define the exponential sum $$F(x;\alpha) := \sum_{n\le x} f(n)\e(\alpha n),$$ as well as $F(x) := F(x; 0) = \sum_{n\le x} f(n).$ For $f=1,$ we note $$E_x(\alpha) := \sum_{1\le n\le x} \e(\alpha n).$$

%We then define the sums on short intervals, $$F(n, n+y; \alpha) := F(n+y;\alpha) - F(n;\alpha) = \sum_{n< m\le n+y} \e(\alpha m)f(m).$$ When $f = \Lambda,$ we note $\psi$ instead of $F$ and when $F(n) = \Ind_{\Omega(n) = k},$ we note $\Pi_k$ instead of $F.$ 

%Finally, when $Q,y < x$ are fixed, we define 
%\begin{equation}\label{def-A}
    %\A = \A(Q,y) := \left\{\frac{a}{q} : q \le Q, (a, q) =1, \left|\alpha - \frac{a}{q}\right| \le \frac{1}{6y} \right\}
%\end{equation}
%the set of rational approximations of $\alpha$ of height $Q$ with precision $1/y,$ which plays a key role in our estimates.

%Throughout the paper, the implicit constants from the $\ll$ and $O()$ notations may depend on constants $\eps$ and $C$ when relevant, but not on $\alpha.$

\subsection{Outline of the proof}\label{sec-main-prop}

Proofs of results such as Theorem \ref{main} rely on estimates of character sums of the form 
\begin{equation}
    \Xi(f, x, Q, \theta) := \sum_{q\le Q} \frac{q}{\phi(q)} \starsum_{\chi \Mod q} \max_{|\lambda|\le\theta} \left|\sum_{n\le x} f(n)\chi(n) \e(\lambda n)\right|.
\end{equation}
(with, in our case, $f = f_k$).
We note here a key difference between the classic Bombieri-Vinogradov setting and the exponential sums setting. While the left-hand side of (\ref{bv-pik}) can be bounded by $Q^{-1}\Xi(f, x, Q,0) (\log x)^{O(1)},$ the left-hand size of (\ref{main-eq}) can only be bounded by $Q^{-1/2} \Xi(f, x, Q,\theta) (\log x)^{O(1)}.$ This is the underlying reason why Bombieri-Vinogradov results on exponential sums only go up to $Q = x^{1/3 - \eps}$ instead of the usual $x^{1/2 - \eps}.$ The reduction to character sums will be performed in Section \ref{reduction}.

We now explain how to obtain bounds on $\Xi(f_k, x, Q,\theta).$ Wolke and Zhan used the Hooley-Huxley contour and zero density estimates, which produces estimates of the form $$\Xi(f_k, x, Q,0) \ll \sup_{1/2\le \sigma\le 1} x^{\sigma+\eps}Q^{6\frac{1-\sigma}{2-\sigma}}.$$
While this is sufficient to obtain $$Q^{-1}\Xi(f_k, x, Q,0) \ll x(\log x)^{-C}$$ for $Q \le x^{1/2-\eps}$ and thus (\ref{bv-pik}), it only proves $$Q^{-1/2}\Xi(f_k, x, Q,0) \ll x(\log x)^{-C}$$ for $Q\le x^{c_0 - \eps}$ where $c_0 := 2\left(2\sqrt{3} - 1\right)^{-2} < 1/3,$ because the supremum is not reached at $\sigma = 1/2$ in this case. This is the reason why Yao's proof of Theorem \ref{main}, which uses the same techniques as Wolke and Zhan, fails above $Q = x^{c_0}.$

Now, in the case of primes, Liu and Zhan \cite{liuzhan} were able to overcome this problem by using (and generalizing to $\theta > 0$) Vaughan's character sum estimate \cite{vaughan-bv}
\begin{equation}\label{vaughan-xi}
    \Xi(\Lambda, x, Q,0) \ll \left(x^{1/2}Q^2 + x^{5/6}Q + x\right)(\log x)^{4},
\end{equation}
which yields better results than zero density estimates around $Q = x^{1/3}.$ For $\theta > 0,$ Liu and Zhan's method leads to the estimate 
\begin{equation}\label{vaughan-xi-theta}
    \Xi(\Lambda, x, Q,\theta) \ll \left(x^{1/2}Q^2 + x^{5/6}Q + x + xQ^2\theta^{1/2}\right)(\log x)^{O(1)},
\end{equation} hence the requirement $\theta \le Q^{-3}(\log x)^{-A}$ to have a bound $\ll xQ^{1/2}(\log x)^{-C}.$
In order to prove Theorem \ref{main}, we prove in Section \ref{bv1} a version of (\ref{vaughan-xi-theta}) for products of $k$ primes, that is
\begin{prop}\label{xi-pik}
    For $2\le Q \le x, \theta\in\R$ and $k \le \log x,$ we have
    $$\Xi(f_k, x, Q,\theta) \ll \left(x^{1/2}Q^2 + x^{5/6}Q + x + xQ^2\theta^{1/2}\right)(\log x)^{O(1)},$$ where the implicit constants are absolute.
\end{prop}
To do this, we first prove (\ref{vaughan-xi-theta}) holds when $\Lambda$ is replaced by any function that can be expressed as a Dirichlet convolution of "Type I" and "Type II" functions, then express $f_k$ as a linear combination of such convolutions (note that in the case of $\Lambda,$ such an expression is given by Vaughan's identity).

%\subsection{Toolbox lemmas}
We end this section by remembering the following elementary analysis estimate, which plays a key role in dealing with the parameter $\theta$ and was already used in Liu and Zhan's proof of (\ref{vaughan-xi-theta}).

\begin{lemma}\label{lambda-term}
    For all $\lambda, x, s = \sigma + \ii\tau$ with $\sigma > 0,$
    $$U(s,x, \lambda) := \int_1^x u^{s-1} \e(\lambda u) \de u \ll \frac{x^\sigma(1 + 1/\sigma)}{1 + \sqrt{\tau} + \tau\Ind_{|\tau| \ge 4\pi|\lambda| x}}.$$
\end{lemma}
This is a consequence of lemmas 3.3/3.4 in Titchmarsh \cite{titchmarsh} after cutting the integral in dyadic intervals. See also \cite[Lemma 8]{liuzhan}.

\section{Reduction to character sums}\label{reduction}

The goal of this section is to derive Theorem \ref{main} from Proposition \ref{xi-pik}. We start by expressing $F(x, a/q)$ in terms of character sums. 

\begin{lemma}\label{decomp}
    Let $a,q$ with $(a, q) = 1.$ Define for $\chi$ a Dirichlet character, $$\chi_{/t} : n \mapsto \Ind_{t|n}\chi(n/t),$$ so that $\left\{\chi_{/t} : t|q, \chi\Mod \frac{q}{t}\right\}$ forms a basis of $q-$periodic functions over $\Z.$ Then
    \begin{equation}\label{decomp-eq}
        F(x; \e_{a/q}) = \frac{1}{\phi(q)}\sum_{st = q} \mu(s) t F(x; t, 0) + O\left(\frac{1}{\phi(q)}\sum_{\substack{q'st = q \\ q'\ne 1}} t\sqrt{q'}\starsum_{\chi\Mod q'} \left|F(x; \chi_{/t})\right| \right).
    \end{equation}
\end{lemma}
\begin{proof}
    Let $r\in \Z/q\Z$ and $d := (r, q).$ Then we have 
    \begin{align*}
        \Ind_{n = r \Mod q} &= \Ind_{d|n}\frac{1}{\phi(q/d)} \sum_{\chi \Mod q/d} \overline{\chi(r/d)}\chi(n/d)
        \\&= \Ind_{d|n}\frac{1}{\phi(q/d)} \sum_{d' | q/d} \starsum_{\chi\Mod q/dd'} \Ind_{(n, dd') = d} \overline{\chi(r/d)} \chi(n/d)
        \\&=\frac{1}{\phi(q/d)} \sum_{d' | q/d} \starsum_{\chi\Mod q/dd'} \sum_{e|d'} \mu(e) \Ind_{de|n} \overline{\chi(r/d)} \chi(n/d)
    \end{align*}
    so that
    \begin{align*}
        \e\left(\frac{a}{q}n\right) &= \sum_{r\Mod q} \e\left(\frac{ar}{q}\right) \Ind_{n = r \Mod q}
        \\&= \sum_{d|q} \sum_{\substack{r \Mod q \\ (r, q) = d}} \frac{1}{\phi(q/d)} \sum_{d' | q/d} \starsum_{\chi\Mod q/dd'} \sum_{e|d'} \mu(e) \Ind_{de|n} \e\left(\frac{ar}{q}\right) \overline{\chi(r/d)} \chi(n/d)
        \\&= \sum_{q'|q} \starsum_{\chi\Mod q'} \sum_{des = q/q'} \frac{\mu(e)\chi(e)}{\phi(q'es)} \Ind_{de|n} \chi(n/de) \sum_{\substack{r\Mod q \\ (r, q) = d}} \e\left(\frac{ar}{q}\right) \overline{\chi(r/d)}
        \\&= \sum_{q'ts = q} \starsum_{\chi\Mod q'} \Ind_{t|n}\chi(n/t) \sum_{de = t} \frac{\mu(e)\chi(e)}{\phi(q'es)} \sum_{r\in (\Z/\frac{q}{d}\Z)^*} \e\left(\frac{ard}{q}\right) \overline{\chi(r)}.
    \end{align*}

    Summing over $n\le x,$ we get \begin{equation}\label{step18}
        F(x; \e_{a/q}) = \sum_{q'ts = q} \starsum_{\chi\Mod q'} F(x ; \chi_{/t}) \sum_{de = t} \frac{\mu(e)\chi(e)}{\phi(q'es)} \sum_{r\in (\Z/\frac{q}{d}\Z)^*} \e\left(\frac{ard}{q}\right) \overline{\chi(r)}.
    \end{equation}

    Here, the inner sum over $r$ is a Gauss sum of the non-primitive character modulo $q/d = q'es$ that derives from $\overline\chi.$ Since $(a, q) = 1,$ we can do the change of variable $r \to ar$ and get $$\sum_{r\in (\Z/\frac{q}{d}\Z)^*} \e\left(\frac{ard}{q}\right) \overline{\chi(r)} = \chi(a) \sum_{r\in (\Z/\frac{q}{d}\Z)^*} \e\left(\frac{rd}{q}\right) \overline{\chi(r)} = \chi(a) \mu(es) \overline{\chi(es)} G(\overline\chi)$$
    (see for example \cite{davenport}, p. 67 for the evaluation of the non-primitive Gauss sum in terms of $G(\overline\chi)$). Thus, when $q'ts = q$ and $\chi$ is a primitive character mod $q',$ we have
    \begin{align*}
        \sum_{de = t} \frac{\mu(e)\chi(e)}{\phi(q'es)} \sum_{r\in (\Z/\frac{q}{d}\Z)^*} \e\left(\frac{ard}{q}\right) \overline{\chi(r)} &= \sum_{de = t} \frac{\mu(e)\chi(e)}{\phi(q'es)} \chi(a) \mu(es) \overline{\chi(es)} G(\overline\chi) 
        \\&= \mu(s)\chi(a)\overline{\chi(s)} G(\overline\chi) \sum_{e|t} \frac{\mu^2(e) \Ind_{(e,s) = 1} |\chi(e)|^2}{\phi(q'es)}
        \\&= \mu(s)\chi(a)\overline{\chi(s)} G(\overline\chi) \prod_{p^\nu\| q's} \frac{1}{p^\nu}\frac{p}{p-1} \prod_{\substack{p|t \\ p\nmid q's}} \frac{p}{p-1}
        \\&= \mu(s)\chi(a)\overline{\chi(s)} G(\overline\chi) \frac{t}{\phi(q)}.
    \end{align*}
    
    Using this in (\ref{step18}), we then get the character decomposition
    \begin{equation}\label{decomp-exact}
        F(x; \e_{a/q}) = \sum_{q'ts = q} \frac{\mu(s)t}{\phi(q)}\starsum_{\chi \Mod q'} \chi(a)\overline{\chi(s)} G(\overline\chi) F(x; \chi_{/t}). 
    \end{equation}

    The main term in (\ref{decomp-exact}) is given by $\chi = 1$ the primitive character modulo $q' = 1,$ for which $F(x; \chi_{/t}) = F(x; t, 0)$ for all $t.$ For other characters, we bound using $|G(\overline\chi)| = \sqrt{q'}$ to obtain the desired result
    \begin{equation*}
        F(x; \e_{a/q}) = \frac{1}{\phi(q)}\sum_{st = q} \mu(s) t F(x; t, 0) + O\left(\frac{1}{\phi(q)}\sum_{\substack{q'st = q \\ q'\ne 1}} t\sqrt{q'}\starsum_{\chi\Mod q'} \left|F(x; \chi_{/t})\right| \right).
    \end{equation*}
\end{proof}

\begin{remark}
    The main term $\frac{1}{\phi(q)}\sum_{st = q} \mu(s) t\sum_{\substack{n\le x \\ t|n}} f(n)$ appearing in (\ref{decomp-eq}) can alternatively be expressed as $$\sum_{st = q}\frac{\mu(s)}{\phi(s)}\sum_{\substack{n\le x \\ (n,q) = t}} f(n),$$ translating the fact the only correlation between $f(n)$ and $n \Mod q$ comes from the value of $(n,q).$ Indeed, we have
    \begin{align*}
        \frac{1}{\phi(q)}\sum_{st = q} \mu(s) t \sum_{\substack{n\le x\\ t|n}} f(n) &= \frac{1}{\phi(q)} \sum_{u|q} F(x; \Ind_{(\cdot, q) =u}) \sum_{\substack{st = q \\ u|t}} \mu(s)t
        \\&= \sum_{ru = q} F(x; \Ind_{(\cdot, q) =u}) \mu(r) \prod_{\substack{p^\nu\| q \\ p^\theta\|r}} \frac{p^{\nu-\theta}}{p^{\nu-1}(p-1)} \prod_{\substack{p^\nu\| q \\ p\nmid r}} \frac{p^{\nu} - p^{\nu-1}}{p^{\nu-1}(p-1)}
        \\&= \sum_{ru = q} F(x; \Ind_{(\cdot, q) =u}) \mu(r) \frac{1}{\phi(r)}.
    \end{align*}
\end{remark}

We also need a Siegel-Walfisz estimate for character sums over products of $k$ primes in order to derive Theorem \ref{main} from Proposition \ref{xi-pik}. 
\begin{lemma}\label{sw-pik}
    There is an absolute constant $c$ such that for $A, \eps > 0,$ $k \le (2-\eps)\log\log x, $ $\lambda \le (\log x)^{-10}$ and $\chi$ a Dirichlet character of conductor $q \ll (\log x)^A,$ we have $$\Pi_k(x; \chi\e_\lambda) \ll_{\eps, A} x\left(q^{1/2}|\lambda|^{1/4-\eps} + \e^{-c\sqrt{\log x}}\right).$$ 
\end{lemma}
The proof is in two parts. If $\lambda \ll \frac{\e^{c\sqrt{\log x}}}{x},$ we have $\Pi_k(x;\chi\e_\lambda) \approx \Pi_k(x;\chi)$ which is controlled by the work of Dupain, Hall and Tenenbaum \cite{dht}; else, $\Pi_k(x;\chi\e_\lambda)$ is a minor arc exponential sum that can be bounded using the work of Granville and Lamzouri \cite{exponential}.
\begin{proof} 
    We first consider the case $\lambda = 0$. We write $f_k(n) = \int_0^1 \e(\alpha \Omega(n)) \e(-\alpha k) \de\alpha,$ so that

    \begin{align*}
        \Pi_k(x; \chi) &= \int_0^1 \e(-\alpha k) \sum_{n\le x} \chi(n)\e(\alpha \Omega(n)) \de\alpha
        \\&\le \sup_{\alpha \in\R} \left|\sum_{n\le x} \chi(n)\e(\alpha \Omega(n))\right| 
        \\&\ll_{A,\eps} x\e^{-c\sqrt{\log x}}
    \end{align*}
    for an absolute constant $c$ by the result of Dupain, Hall and Tenenbaum \cite[Lemma 2]{dht} when $\chi$ is of conductor $q \ll (\log x)^A$ and $k\le (2-\eps)\log\log x.$

    Then, if $|\lambda| \le \frac{\e^{\frac{c}{2}\sqrt{\log x}}}{x},$ we may integrate by parts
    \begin{align*}
        \Pi_k(x;\chi\e_\lambda) &= \sum_{n\le x} f_k(n) \chi(n) \left(\e(\lambda x)- \int_n^x \e(u\lambda) 2\ii\pi\lambda \de u\right)
        \\&= \Pi_k(x; \chi) \e(\lambda x) - 2\ii\pi\lambda \int_0^x \e(u\lambda) \Pi_k(u; \chi) \de u
        \\&\ll (1+|\lambda| x) x\e^{-c\sqrt{\log x}} \ll x\e^{-\frac{c}{2}\sqrt{\log x}}.
    \end{align*}
    If $|\lambda| > \frac{\e^{\frac{c}{2}\sqrt{\log x}}}{x},$ we instead decompose $\chi(n) = \frac{1}{G(\overline\chi)}\sum_{(a,q) = 1} \overline{\chi}(a) \e\left(\frac{a}{q}n\right),$ so that
    \begin{equation}\label{step45}
        \Pi_k(x;\chi\e_\lambda) = \frac{1}{G(\overline\chi)}\sum_{(a,q) = 1} \int_0^1 \e(-\alpha k) \sum_{n\le x} \e\big(\alpha \Omega(n)\big) \e\left(\left(\frac{a}{q}+\lambda\right)n\right) \de \alpha.
    \end{equation}
    We now write for $a$ coprime with $q,$ $$\lambda_a := \frac{a}{q}+\lambda,\, L := \min\left\{\frac{1}{2|\lambda|}, \e^{\frac{c}{2}\sqrt{\log x}}\right\},\, Q := x/L,\, r := \min\left\{r: \exists(b,r) = 1, \left|\lambda_a - \frac{b}{r}\right| \le \frac{1}{rQ}\right\} \le Q.$$ 
    Note that $\frac{b}{r}\ne \frac{a}{q}$ since $|\lambda| > \frac{\e^{\frac{c}{2}\sqrt{\log x}}}{x}\ge 1/Q.$ Therefore, $\frac{1}{rq} \le \left|\frac{a}{q}- \frac{b}{r}\right| \le |\lambda| + \frac{1}{rQ},$ so that $r \ge \frac{1}{|\lambda|}\frac{qQ}{Q-q} \ge L.$ 
    
    We thus have $L \le r\le x/L$ with $\left|\lambda_a - \frac{b}{r}\right| \le \frac{1}{rQ} \le \frac{1}{x}$ and $L \gg (\log x)^{10}$ and may thus apply the result of Granville and Lamzouri \cite[Theorem 1.1]{exponential} with $M := L^{1/2}/\log x$ to get for all $\alpha,$
    \begin{equation}\label{step46}
        \sum_{n\le x} \e\big(\alpha\Omega(n)\big) \e(\lambda_a n) = \sum_{m\le M} \e\big(\alpha\Omega(m)\big) \sum_{x/M \le p \le x/m} \e(\alpha) \e(\lambda_a mp) + O\left(\frac{x}{L^{1/4}}\right).
    \end{equation}
    Now, for given $m,$ using the fact that $\lambda_a m = \frac{am}{q} + \lambda m$ with $Lm/x \le |\lambda m| \le m/L,$ we may construct (analogously to $\frac{b}{r}$ above) $\frac{b_m}{r_m}$ with $\left|\lambda_a m - \frac{b_m}{r_m}\right| \le \frac{1}{x}$ and $\frac{L}{m}\le r_m \le x\frac{m}{L}.$
    Thus, by the classical result of Vinogradov \cite[§9]{vinogradov}, we have for $\eps > 0$ and $m\le M = L^{1/2}/\log x,$ $$\sum_{x/M \le p \le x/m} \e(\lambda_a mp) \ll \frac{x}{m (L/m)^{1/2-\eps}}$$
    and thus by (\ref{step46})
    \begin{align*}
        \left|\sum_{n\le x} \e\big(\alpha\Omega(n)\big) \e(\lambda_a n)\right| &\le \sum_{m\le M} \left|\sum_{x/M \le p \le x/m} \e(\lambda_a mp)\right| +  O\left(\frac{x}{L^{1/4}}\right)
        \\&\ll \sum_{m\le M} \frac{x}{m (L/m)^{1/2-\eps}} +  \frac{x}{L^{1/4}}
        \\&\ll x\left(\frac{M}{L}\right)^{1/2-\eps} + \frac{x}{L^{1/4}} \ll \frac{x}{L^{1/4-\eps}}.
    \end{align*}
    Injecting that bound in (\ref{step45}), we obtain the desired bound
    $$\Pi_k(x;\chi\e_\lambda) \ll q^{1/2}\frac{x}{L^{1/4-\eps}}.$$
\end{proof}

We can now prove Theorem \ref{main} using Proposition \ref{xi-pik}, which will be proven in the next section.
\begin{proof}[Proof of Theorem \ref{main}]
    We apply Lemma \ref{decomp} to $f = f_k\e_\lambda$ to bound the left-hand side of (\ref{main-eq}) is bounded by
    \begin{align}
        \ll E :&= \sum_{q\le Q} \frac{1}{\phi(q)} \sum_{\substack{q'st = q \\ q'\ne 1}} t\sqrt{q'}\starsum_{\chi\Mod q'} \max_{|\lambda|\le \theta} \left|\Pi_k(x; \chi_{/t}\e_{\lambda})\right| \notag
        \\&= \sum_{t\le Q} t \sum_{1 < q'\le Q/t} \sqrt{q'} \starsum_{\chi\Mod q'}  \sum_{s\le Q/q't} \frac{1}{\phi(q'st)} \max_{|\lambda|\le \theta} \left|\Pi_k(x; \chi_{/t}\e_{\lambda})\right| \notag
        \\&\le \log Q \sum_{t\le Q} \frac{t}{\phi(t)} \sum_{1 < q'\le Q/t} \frac{\sqrt{q'}}{\phi(q')} \starsum_{\chi\Mod q'}  \max_{|\lambda|\le \theta} \left|\Pi_{k-\Omega(t)}(x/t; \chi\e_{\lambda t})\right|, \label{step99}
    \end{align}
    noting that we always have $\phi(q'st) \ge \phi(q')\phi(s)\phi(t)$ and $$\Pi_k(x; \chi_{/t}\e_\lambda) = \sum_{\substack{n \le x \\t|n \\ \Omega(n) = k}} \chi(n/t) \e(\lambda n) = \sum_{\substack{m\le x/t \\ \Omega(m) = k - \Omega(t)}} \chi(m) \e(\lambda tm) = \Pi_{k-\Omega(t)}(x/t;\chi\e_{\lambda t}).$$

    We now split the sum over $q'$ in (\ref{step99}) into dyadic intervals. For all $1\le t\le Q$ and $Q' \le Q/t,$ we have by Proposition \ref{xi-pik}
    \begin{align}
        \sum_{Q' < q'\le 2Q'} \frac{\sqrt{q'}}{\phi(q')} \starsum_{\chi\Mod q'} \max_{|\lambda|\le \theta} \left|\Pi_{k-\Omega(t)}(x/t; \chi\e_{\lambda t})\right| &\le Q'^{-1/2} \log Q' \big(\Xi(f_{k-\Omega(t)}, x/t,2Q', \theta t) - \Xi(f_{k-\Omega(t)}, x/t,Q', \theta t)\big) \notag\\&\ll \left(\frac{x^{1/2}Q'^{3/2}}{t^{1/2}} + \frac{x^{5/6}Q'^{1/2}}{t^{5/6}} + \frac{xQ'^{-1/2}}{t} + \frac{xQ'^{3/2}\theta^{1/2}}{t^{1/2}}\right) (\log x)^A \notag 
        \\&\ll \left(\frac{x^{1/2}Q^{3/2}}{t^{2}} + \frac{x^{5/6}Q^{1/2}}{t^{4/3}} + \frac{xQ'^{-1/2}}{t} + \frac{xQ^{3/2}\theta^{1/2}}{t^{2}}\right) (\log x)^A \notag
        \\&\ll \frac{x}{t} (\log x)^A \left(Q'^{-1/2} + (\log x)^{-B/2}\right)
        \label{step100}
    \end{align} 
    where $A$ is an absolute constant and we used the conditions $Q \le x^{1/3}(\log x)^{-B},$ $\theta \le Q^{-3}(\log x)^{-B}.$

    Thus, for $Q' \ge (\log x)^{B/7},$ the left-hand side of (\ref{step100}) is $\ll \frac{x}{t}(\log x)^{A-B/14}.$ 
    Moreover, when $Q'\le (\log x)^{B/7},$ we also have by Lemma \ref{sw-pik} with $\theta \le (Q't)^{-3}(\log x)^{-B}$
    \begin{align*}
        \sum_{Q' < q'\le 2Q'} \frac{\sqrt{q'}}{\phi(q')} \starsum_{\chi\Mod q'} \max_{|\lambda|\le \theta} \left|\Pi_{k-\Omega(t)}(x/t; \chi\e_{\lambda t})\right| &\ll Q'^{3/2} \frac{x}{t}\left(Q'^{1/2}(\theta t)^{1/4-\eps} + \e^{-c\sqrt{\log x}}\right)
        \\&\ll  \frac{x}{t}\left((\log x)^{-B/14 + \eps} + (\log x)^{B/4}\e^{-c\sqrt{\log x}}\right)
        \\&\ll \frac{x}{t}(\log x)^{A-B/14}.
    \end{align*}
    %shows the left-hand side of (\ref{step100}) is also $\ll Q'^{3/2}\frac{x}{t}\e^{-c\sqrt{\log x}} \ll \frac{x}{t}(\log x)^{A-B/2}$ when $Q'\le (\log x)^B.$ %\theta t \le Q^{-2}  
    Thus, we can inject this bound in (\ref{step99}) to get 
    $$E \ll (\log Q)^2 \sum_{t\le Q} \frac{t}{\phi(t)} \frac{x}{t} (\log x)^{A-B/2} \ll x(\log x)^{3+A-B/14}$$ as desired fixing $B = 14(A+C+3).$
\end{proof}

\section{A generalised Vaughan estimate}\label{bv1}

The main result of this part is a generalised character sum estimate for a large class of functions, given below. We then use this result to prove Proposition \ref{xi-pik}.

\begin{theorem}\label{main-bv}
    Assume $f$ on $[1,x]$ can be expressed as a Dirichlet convolution $f_{|[1,x]} = (g_1 * g_2 * \ldots * g_r)_{|[1,x]},$ where each $g_i$ is log-bounded and verifies either of the following:
    \begin{itemize}
        \item either $g_i$ verifies $g_i(n+1) - g_i(n) \ll 1/n$ for $1\le n\le x$ (we then say $g_i$ is a type I function);
        \item or $g_i$ is supported on $\left[1, x^{2/3}\right]$ (we then say $g_i$ is a type II function). 
    \end{itemize}

    Then $$\Xi(f, x, Q, \theta) \ll \left(x^{1/2}Q^2 + x^{5/6}Q + x + xQ^2\theta^{1/2}\right)(\log x)^{O(r)}.$$
\end{theorem}
\begin{remark}
    Given the expression of $\Xi(f, x, q, \theta),$ the conclusion also holds if $f$ can be expressed as a linear combination of such convolutions. For example, for $f = \Lambda,$ Vaughan's identity $\Lambda = \Lambda_{|[1, x^{1/2}]} + \mu_{|[1, x^{1/2}]} * \log - \mu_{|[1, x^{1/2}]} * 1 * \Lambda_{|[1, x^{1/2}]}$ lets us retrieve (\ref{vaughan-xi-theta}) from Theorem \ref{main-bv}.
\end{remark}
We prove Theorem \ref{main-bv} using Vaughan's method \cite{vaughan-bv}. This requires estimates on sums of the form $$\sum_{m,n} a(m) b(n) \chi(mn) \e(\lambda mn),$$ which are given by the two lemmas below.

\begin{lemma}\label{type-I}
    Let $g$ such that $g(n+1) - g(n) \ll 1/n,$ $\chi$ a primitive Dirichlet character of conductor $q,$ and $|\lambda|\le 1/2q.$ Then $$\sum_{n\le x} g(n) \chi(n) \e(\lambda n) \ll q^{1/2} \log q \log x.$$
\end{lemma}
\begin{proof}
    We can decompose $\chi$ as a sum of exponential terms, $$\chi(n) =  \frac{1}{G(\overline{\chi})} \sum_{(a,q) = 1} \overline{\chi}(a) \e\left(\frac{an}{q}\right),$$ where $G(\overline{\chi})$ is the Gauss sum verifying $|G(\overline\chi)| = q^{1/2}.$ We thus have 
    \begin{align*}
        \left|\sum_{n\le x} g(n) \chi(n) \e(\lambda n)\right| &\le q^{-1/2} \sum_{(a,q) = 1} \left|\sum_{n\le x} g(n) \e\big((a/q+\lambda) n\big)\right| \\&\ll q^{-1/2} \sum_{(a,q) = 1} \frac{1}{\|a/q + \lambda\|} \left(|g(1)| + \sum_{n\le x} \big|g(n+1) - g(n)\big|\right)
        \\&\ll q^{1/2} \log q\log x,
    \end{align*}
    integrating by parts and using the hypotheses on $\lambda$ and $g.$
\end{proof}

\begin{lemma}\label{type-II}
    Let $M\le M', N\le N'$ with $MN \asymp x \asymp M'N',$ and $a, b$ supported on $[M,M'], [N,N']$ respectively with $$\sum_{M\le m \le M'} \frac{|a(m)|^2}{m}, \sum_{N\le n \le N'} \frac{|b(n)|^2}{n} \ll (\log x)^C$$ for some $C.$ 
    Then for all $\theta,$ $$\Xi(a * b, x, Q, \theta) %= \sum_{q\le Q} \starsum_{\chi\bmod q} \max_{|\lambda|\le \theta} \left|\sum_{mn\le x} a(m) b(n) \chi(mn) \e(\lambda mn)\right| 
    \ll \left(x + x^{1/2}Q(M+N)^{1/2} + x^{1/2}Q^2 + xQ^2\theta^{1/2}\right)(\log x)^{C+1}.$$
\end{lemma}
\begin{proof}
    This is essentially equation (4.12) in Liu and Zhan \cite{liuzhan}; we provide here the main steps of the proof.
    
    We integrate by parts $\e(\lambda mn)$ to reduce to sums of the form $\sum_{mn\le u} a(m) b(n) \chi(mn),$ and then use Perron's formula to handle the $mn \le u$ condition. We then obtain
    $$\sum_{mn\le x} a(m) b(n) \chi(mn) \e(\lambda mn) = \int_{\sigma - \ii x^{1+\sigma}}^{\sigma + \ii x^{1+\sigma}} \frac{1}{2\ii\pi} \left(\int_0^x \e(u\lambda) u^{s-1} \de u\right) \cdot \sum_m \frac{a(m)\chi(m)}{m^s} \sum_n \frac{b(n)\chi(n)}{n^s} \de s + O(1)$$
    for $|\lambda|\le\theta$ and $\sigma > 0$ (since $a$ and $b$ have finite support, there is no concern of convergence).

    Taking the max over $\lambda,$ summing on $q$ and $\chi$ and using Cauchy-Schwarz, we then get 
    \begin{equation}\label{step7}
        \Xi(a * b, x, Q, \theta) \le \int_{\sigma - \ii x^{1+\sigma}}^{\sigma + \ii x^{1+\sigma}} \frac{1}{2\ii\pi} \max_{|\lambda|\le\theta}\left|\int_0^x \e(u\lambda) u^{s-1} \de u\right| \cdot S_a(Q,s)^{1/2} S_b(Q,s)^{1/2} \de s + O(Q^2),
    \end{equation}
    where $$S_a(q,s) := \sum_{q\le Q} \starsum_{\chi\bmod q} \left|\sum_m \frac{a(m)\chi(m)}{m^s}\right|^2$$ and $S_b(q,s)$ is defined analogously.

    Now, we have by Montgomery \cite[Theorem 7.1]{montgomery} that for all $T,$ $$\int_T^{2T} S_a(q,\sigma + it) \de t \ll (M' + Q^2T) \sum_m \frac{|a(m)|^2}{m^{2\sigma}} \ll M^{1-2\sigma}(M + Q^2T)(\log x)^C.$$ 
    Picking $\sigma = 1/2,$ bounding the integral over $u$ by Lemma \ref{lambda-term} and using Cauchy-Schwarz, we then get for all $T$
    \begin{align*}
        \int_{\sigma + \ii T}^{\sigma + \ii 2T} \frac{1}{2\ii\pi} \max_{|\lambda|\le\theta} \left|\int_0^x \e(u\lambda) u^{s-1} \de u\right| \cdot S_a(Q,s)^{1/2} S_b(Q,s)^{1/2} \de s &\ll \frac{x^{1/2} (M + Q^2T)^{1/2} (N + Q^2T)^{1/2}}{1+ T^{1/2} + T\Ind_{T\ge 4\pi\theta x}}(\log x)^C
        \\&\ll \frac{x + x^{1/2}T^{1/2}Q(M+N)^{1/2} + x^{1/2}Q^2T}{1+ T^{1/2} + T\Ind_{T\ge 4\pi\theta x}}(\log x)^C,
    \end{align*}
    and we can get the same bound for the central interval $\Im(s) \in [-1,1].$ Thus, coming back to (\ref{step7}) and summing on dyadic intervals, we have 
    $$\Xi(a * b, x, Q, \theta) \ll \left(x + x^{1/2}Q(M+N)^{1/2}\log(1+\theta x) + x^{1/2}Q^2\left((\theta x)^{1/2} + \log x\right)\right)(\log x)^C.$$
\end{proof}

We can now prove Theorem \ref{main-bv}.

\begin{proof}[Proof of Theorem \ref{main-bv}]
    The proof follows Vaughan's ideas in \cite{vaughan-bv}, applied to the more general setting of Theorem \ref{main-bv}. We fix $y\le \min\left\{x^{1/3}, 1/2Q\theta\right\}$ and for any primitive Dirichlet character $\chi$ of conductor $q$ and $\lambda \le \theta,$ we split the sum 
    \begin{align}
        \sum_{n\le x} f(n) \chi(n) \e(\lambda n) &= \sum_{\prod_i n_i \le x} \prod_{i=1}^r g_i(n_i) \chi(n_i) \e\left(\lambda \prod_{i=1}^r n_i\right) \notag
        \\&= \sum_{j=1}^r \sum_{\substack{\prod_i n_i \le x \\ n_j > x/y}} \prod_{i=1}^r g_i(n_i) \chi(n_i) \e\left(\lambda \prod_{i=1}^r n_i\right) + \sum_{\substack{y < \prod_i n_i \le x \\ \forall i, n_i \le x/y}} \ldots + \sum_{\prod_i n_i \le y} \prod_{i=1}^r \ldots \notag
        \\&=: \sum_{j=1}^r S_j(\chi,\lambda) + S_{II}(\chi,\lambda) + O\left(y(\log x)^{O(r)}\right), \label{sum-decomp}
    \end{align}
    noting that two of the $n_i$ cannot simultaneously be greater than $x/y$ under the condition $\prod_i n_i \le x$ as $y\le x^{1/2}.$ 

    We first consider $S_j(\chi,\lambda).$ If $g_j$ is a type II function, then $S_j(\chi) = 0$ as $g_j(n_j) = 0$ for $n_j > x/y.$ If $g_j$ is a type I function, we write 
    \begin{align*}
        S_j(\chi) &= \sum_{\prod_{i\ne j} n_i < y} \prod_{i\ne j} g_i(n_i)\chi(n_i) \sum_{x/y < n_j \le x/\prod_{i\ne j} n_i} g_j(n_j) \chi(n_j) \e\left(\lambda \prod_{i\ne j}n_i \cdot n_j \right)
        \\&\ll \sum_{\prod_{i\ne j} n_i < y} \prod_{i\ne j} |g_i(n_i)| q^{1/2} \log q \log x
        \\&\ll y q^{1/2} (\log x)^{O(r)}
    \end{align*}
    where we used Lemma \ref{type-I} since $g_j$ is type I and $|\lambda| \prod_{i\ne j} n_i \le \theta y \le 1/2Q.$
    Thus, \begin{equation}\label{sum-I}
        S_j(\chi,\lambda) \ll y q^{1/2} (\log x)^{O(r)}
    \end{equation} holds for all $\chi, j, |\lambda| \le \theta.$

    We now consider $S_{II}(\chi,\lambda).$ We decompose the sum over $n_i$ into dyadic rectangles $N_i < n_i \le 2N_i$ with $y < \prod_i N_i \le x$ and $\forall i, N_i \le x/y$ and note $S_{II}(\chi,\lambda, N_1, \ldots, N_r)$ the corresponding sum.

    Once $N_1, \ldots, N_r$ are fixed, we define a subset $\I \subset \{1,\ldots, r\}$ as follows.
    \begin{itemize}
        \item if at least one of the $N_i$ is larger than $y,$ let $\I = \{\arg\max N_i\}$ a singleton, so that $$y \le \prod_{i\in \I} N_i = \max_{1\le i\le r} N_i \le x/y$$
        as all $N_i$ are smaller than $x/y;$
        \item else, let $\ell$ the smallest index such that $\prod_{i=1}^\ell N_i \ge y,$ and $\I = \{1,\ldots, \ell\}.$ This way, since all the $N_i$ are smaller than $y$ in this case we have $$y \le \prod_{i\in\I} N_i \le y^2 \le x/y$$ as $y\le x^{1/3}.$
    \end{itemize}

    Noting $A = \prod_{i\in\I} N_i$ and $B = \prod_{i\not\in\I} N_i,$ we can then write 
    $$S_{II}(\chi, \lambda, N_1, \ldots, N_r) = \sum_{\substack{A\le a \le 2^rA \\ B\le b\le 2^rB \\ ab \le x}} h_1(a) h_2(b) \chi(ab) \e(\lambda ab),$$
    where $h_1 = \ast_{i\in\I} \left(g_i\Ind_{(N_i, 2N_i]}\right)$ and $h_2 = \ast_{i\not\in\I} \left(g_i\Ind_{(N_i, 2N_i]}\right)$ depend only on the $N_i$ and are supported on $[A, 2^rA]$ and $[B, 2^rB]$ respectively.

    These type II sums can be estimated on average over $\chi$ by Lemma \ref{type-II}. Once $N_1,\ldots, N_r$ (and thus $A, B, h_1, h_2$) are fixed, we thus have
    \begin{align}
        \sum_{q\le Q} \starsum_{\chi\Mod q} \max_{|\lambda|\le\theta} \big|S_{II}(\chi,\lambda, N_1, \ldots, N_r)\big| &= \Xi(h_1 * h_2, x, Q, \theta) \notag 
        \\&\ll \left(x + x^{1/2}Q(A+B)^{1/2} + x^{1/2}Q^2 + xQ^2\theta^{1/2}\right)(\log x)^{O(r)} \notag
        \\&\ll \left(x + xQ/y^{1/2} + x^{1/2}Q^2 + xQ^2\theta^{1/2}\right)(\log x)^{O(r)}\label{sum-II}
    \end{align}
    since by construction $AB\le x, y\le A\le x/y$ and thus $B\le x/A \le x/y.$ Summing over $\ll (\log x)^r$ dyadic rectangles $N_1,\ldots, N_r$ then gives the same bound for $\sum_q \frac{q}{\phi(q)} \sum_{\chi\Mod q}^* \max_{|\lambda|\le\theta} \big|S_{II}(\chi,\lambda)\big|.$

    Summing over $\chi$ in (\ref{sum-decomp}) using (\ref{sum-I}) and (\ref{sum-II}), we then get $$\Xi(f,x, Q, \theta)\ll \left(yQ^{5/2} + x^{1/2}Q^2 + x + xQ/y^{1/2} + xQ^2\theta^{1/2}\right) (\log x)^{O(r)}$$ for all $y\le y_{\max} := \min\left\{x^{1/3}, 1/2Q\theta\right\}.$ The conclusion then follows from taking $y = \min\left\{x^{2/3}/Q, y_{\max}\right\}.$
\end{proof}

We now consider the specific case of $f_k(n) = \Ind_{\Omega(n) = k}$ and decompose it in order to deduce Proposition \ref{xi-pik} from Theorem \ref{main-bv}.

\begin{proof}[Proof of Proposition \ref{xi-pik}]
    %Following Drappeau-Topacogullari \cite[part 3.2]{drappeau},
    We start by reducing the problem to functions supported on $x^{1/2}$-friable integers. To do this, we write $f^+(n) = \Ind_{P^-(n) > x^{1/2}}$ and $f^-(n) = \Ind_{P^+(n) \le x^{1/2}}$ the indicator functions of $x^{1/2}$-rough and $x^{1/2}$-friable integers respectively,  and note the equality between multiplicative functions
    $$1 * (\mu f^-) = f^+$$ which is easily checked on prime powers. Now, for $n\le x,$ $n$ has either zero or one prime factors greater than $x^{1/2},$ so that
    \begin{align*}
        f_k(n) &= \sum_{\substack{ab = n \\ P^+(a) \le x^{1/2} \\ P^-(b) > x^{1/2}}} f_k(a) \Ind_{b=1} + f_{k-1}(a) (1 -\Ind_{b=1})
        \\&= f_k f^- (n) + f_{k-1} f^- * f^+ (n) - f_{k-1} f^- (n).
    \end{align*}
    Using the identity $1 * (\mu f^-) = f^+,$ we thus have the identity
    \begin{equation}\label{step8}
        (f_k)_{|[1,x]} = \big(f_k f^- + (f_{k-1} f^-) * (\mu f^-) * 1 - f_{k-1} f^-\big)_{|[1,x]}. 
    \end{equation}
    In order to use Theorem \ref{main-bv}, it is then sufficient to split $f_k f^-$ and $\mu f^-$ into functions supported in $[1, x^{2/3}].$ To do this, for $x^{1/3} < n \le x$ with $P^+(n) \le x^{1/2},$ we write its prime factorisation $n = p_1 \ldots p_r$ with $p_1 \ge p_2 \ge \ldots \ge p_r,$ and consider the lowest index $\ell$ such that $p_1\ldots p_\ell > x^{1/3}.$ Then, $a = p_1\ldots p_\ell$ and $b = p_{\ell+1}\ldots p_r$ are the only $a$ and $b$ verifying all the conditions $$ab = n ; P^-(a) \ge P^+(b) ; x^{1/3} < a \le x^{1/3}P^{-}(a).$$ Moreover, we always have $b = n/a \le x^{2/3}$ and $a\le x^{2/3}$ (as either $p_1 > x^{1/3}$ in which case $a = p_1 = P^+(n) \le x^{1/2},$ or $p_1 \le x^{1/3}$ in which case $a = x^{1/3}p_\ell \le x^{1/3} p_1 \le x^{2/3}$). Thus, we may write for $n\le x$ and an arithmetic function $g,$
    \begin{equation}\label{step9}
        gf^{-}(n) = \sum_{ab = n} \Ind_{a \le x^{2/3}} \Ind_{x^{1/3} < a \le x^{1/3}P^{-}(a)} f^-(a) \Ind_{b \le x^{2/3}} g(ab) \Ind_{P^-(a) \ge P^+(b)}.
    \end{equation}
    This is almost a convolution of two bounded functions supported on $[1, x^{2/3}];$ we still need to separate the variables in the term $g(ab) \Ind_{P^-(a) \ge P^+(b)},$ where $g = \mu$ or $g = f_k.$

    For $g = \mu,$ we write $$\mu(ab) \Ind_{P^-(a) \ge P^+(b)} = \mu(a) \mu(b) \Ind_{P^-(a) > P^+(b)}$$ (as, when $P^-(a) \ge P^+(b),$ $a$ and $b$ share a common factor if and only if $P^-(a) = P^+(b).$) We may then use \cite[Lemma 13.11]{ik} to write $$\Ind_{P^-(a) > P^+(b)} = \int_{-\infty}^{+\infty} h(t) P^{-}(a)^{\ii t} P^+(b)^{-\ii t} \de t,$$ valid for all $a,b\le x^{2/3}$ with a function $h$ depending only on $x$ and verifying $\int_{-\infty}^{+\infty} |h(t)| dt \ll \log x.$ Thus, for $g = \mu,$ (\ref{step9}) becomes 
    \begin{equation}\label{mu-split}
        \mu f^-(n) = \int_{-\infty}^{\infty} h(t) \sum_{ab = n} g_1(a,t) g_2(b,t) 
    \end{equation}
    for some functions $g_1, g_2$ verifying $|g_i(a,t)| \le 1$ and $g_i(a,t) = 0$ if $a> x^{2/3}$ for all $i\in\{1,2\}$ and $t\in \R.$

    For $g = f_k,$ we write \begin{equation}\label{step10}
        f_k(ab) = \sum_{\ell = 0}^k f_\ell (a) f_{k-\ell}(b),
    \end{equation}
    and again separate the variables in $\Ind_{P^-(a) \ge P^+(b)}$ using \cite[Lemma 13.11]{ik} to obtain from (\ref{step9}) an expression of $f_k f^-(n)$ similar to (\ref{mu-split}), albeit with an extra sum over $\ell.$ Injecting these expressions in (\ref{step8}), we may then %apply the linear operator $\Xi(\cdot, x, Q,\lambda)$ to (\ref{step8}) and bound all the terms in the right-hand side 
    use Theorem~\ref{main-bv} on each term to get $$\Xi(f_k, x, Q,\theta) \ll \left(x^{1/2}Q^2 + x^{5/6}Q + x + xQ^2\theta^{1/2}\right)(\log x)^{O(1)}$$ as desired. The sum over $\ell$ coming from (\ref{step10}) does cost us an extra $O(k)$ factor and the integrals weighted by $h(t)$ each cost us an extra $O(\log x)$ factor, but they are absorbed by the $(\log x)^{O(1)}$ factor.
\end{proof}
\begin{remark}
    When dealing with $h_k = \Ind_{\omega(n) = k}$ instead of $f_k,$ we again use that when $P^-(a) \ge P^+(b),$ the only possible common factor of $a$ and $b$ would be $P^-(a) = P^+(b).$ We therefore have 
    $$h_k(ab) \Ind_{P^-(a) \ge P^+(b)} = \sum_{\ell = 0}^k f_\ell (a) f_{k-\ell}(b) \Ind_{P^-(a) > P^+(b)} + \sum_{\ell = 0}^k f_\ell (a) f_{k+1-\ell}(b) \Ind_{P^-(a) = P^+(b)}$$ and can proceed again with the decomposition.
\end{remark}

%--
\section{Application to lower bounds}\label{end}

The goal of this final section is to prove Theorem \ref{uniform-pik} using Theorem \ref{main} and a result from the author's previous paper \cite{self}, reminded below.

\begin{theorem}[\cite{self}, Proposition 10]\label{y-le-q}
    Assume $f$ is supported on $[1,x],$ fix $Q \le x^{1/2}$ and assume there is a constant $C$ and functions $g$ and $\eta$ (which may depend on $x$) with $\eta(q) \ll (\log q)^C$ such that 
    \begin{equation}\label{bv-e}
        \sum_{q\le Q} |g(q)| \max_{(a,q) = 1} \left|F(x; \e_{a/q}) - x\frac{g(q)}{\phi(q)} \right| \ll \frac{x}{(\log x)^{C}}
    \end{equation} and for all $\delta > 0,$
    \begin{equation}\label{goal-g}
        \sum_{Q^{1-\delta}< q \le Q} \frac{|g(q)|^2}{\phi(q)} \gg \frac{\log Q}{\eta(Q)}.
    \end{equation}
    Then for $\eps > 0$ and $y \le Q^{1-\eps},$ we have $$\frac{1}{x}\sum_{-y< n\le x} \big|F(n+y;\e_\alpha) - F(n;\e_\alpha)\big|^2 \gg \frac{y\log Q}{\eta(Q)}.$$
\end{theorem}

Here, (\ref{bv-e}) will be deduced from Theorem \ref{main} with $\theta = 0.$ We thus introduce the normalized main term appearing in Theorem \ref{main},
\begin{equation}\label{def-g}
    g_{k,x}(q) := \frac{1}{\Pi_k(x)} \sum_{st = n} \mu(s)t \#\{n \le x : t|n, \Omega(n) = k\} = \frac{1}{\Pi_k(x)} \sum_{st = n} \mu(s)t \Pi_{k-\Omega(t)}(x/t).
\end{equation}
Since we only need the lower bound (\ref{goal-g}) on $\sum_q {\frac{|g_{k,x}(q)|^2}{\phi(q)}}$, we may only count the terms where $q$ is prime, which simplifies the expression of $g_{k,x}(q).$ Indeed, when $q=p$ is prime, the sum in (\ref{def-g}) contains only two terms, so that \begin{equation}
        g_{k,x}(p) = \frac{p\Pi_{k-1}(x/p)}{\Pi_k(x)} - 1 \label{q-prime}.
    \end{equation}
In this case, we obtain the following estimates on $g_{k,x}(p).$

\begin{lemma}\label{estim-pik}
    Fix $\eps > 0.$ When $p \le x^{1-\eps}$ is prime and $k \le (2-\eps) \log\log x,$ we have \begin{equation}\label{g-far}
        g_{k,x}(p) = \left( \frac{k}{\log\log x}(1-t)^{\frac{k-\log\log x}{\log\log x}} - 1 + O\left(\frac{1}{\log\log x}\right)\right)
    \end{equation} where $t := \frac{\log p}{\log x}.$

    Moreover when $|k-\log\log x| \le C$ for some fixed $C,$ we have with $\ell := \log\log x - \log\log (x/p) = -\log(1-t),$
    \begin{equation}\label{g-near}
        g_{k,x}(p) = \left(\frac{(k-\log\log x - 1 - H'(1))(1-\ell) + \ell - \ell^2/2}{\log\log x} + O\left(\frac{1}{(\log\log x)^2}\right) \right)
    \end{equation}
    where $H$ has been defined in (\ref{def-H}).
\end{lemma}

The first estimate (\ref{g-far}) will be derived from the Sathe-Selberg estimates on $\Pi_k$ (Lemma \ref{sathe-selberg}). However, it is not sufficient to prove Theorem \ref{uniform-pik} near $k = \log\log x,$ where the main term in (\ref{g-far}) vanishes. Fortunately, Selberg's method allows to precise the $O\left(\frac{1}{\log\log x}\right)$ term appearing in the expression of $\Pi_k(x)$ (see \cite[Theorem II.6.5]{tenen-book}). In particular, when $|k - \log\log x| \le C$ for some fixed $C,$ we have (\cite[equation (34)]{balazard})
\begin{equation}\label{near-ll}
    \Pi_k(x) = \frac{x}{\log x}\frac{(\log\log x)^{k-1}}{(k-1)!}\cdot\left(1 + \frac{(k-1-\log \log x) H'(1) - H''(1)/2}{\log\log x} + O\left(\frac{1}{(\log\log x)^2}\right)\right),
\end{equation}
which we use to establish the estimate (\ref{g-near}).

\begin{proof}[Proof of Lemma \ref{estim-pik}]
    Noting $t := \frac{\log p}{\log x} \le 1 - \eps, L := \log\log x$ and $\ell := -\log(1-t) \ll 1,$ we have by Lemma~\ref{sathe-selberg}
    \begin{align*}
        p\Pi_{k-1}(x/p) &= \frac{x}{\e^{-\ell} \log x} \frac{(L-\ell)^{k-2}}{(k-2)!} \cdot\left(H\left(\frac{k-2}{L-\ell}\right) + O\left(\frac{1}{L}\right)\right)
        \\&= \frac{x}{\log x} \frac{L^{k-2}}{(k-2)!} \e^{\ell - (k-2)\frac{\ell}{L} + O\left(\frac{k}{L^2}\right)} \left(H\left(\frac{k-1}{L}\right) + O\left(\frac{1}{L}\right)\right)
        \\&= \Pi_k(x) \frac{k-1}{L} \e^{-\ell \frac{k-L}{L}} \left(1 + O\left(\frac{1}{L}\right)\right)
    \end{align*}
    since $H'\left(\frac{k-1}{L}\right) \ll 1$ for $k \le (2-\eps)L,$
    so that by formula (\ref{q-prime}), $$g_{k,x}(p) = \frac{k}{L} \e^{-\ell \frac{k-L}{L}} - 1 + O\left(\frac{1}{L}\right)$$ as desired. 

    We now consider the case $k = L + O(1),$ in which case the main term above vanishes. In this case, it follows from estimate~(\ref{near-ll})
    \begin{align*}
        p\Pi_{k-1}(x/p) &= \frac{x}{\e^{-\ell} \log x} \frac{(L-\ell)^{k-2}}{(k-2)!} \cdot\left(1 + \frac{(k-2 - L + \ell) H'(1) - H''(1)/2}{L - \ell} + O\left(\frac{1}{L^2}\right)\right)
        \\&= \frac{x}{\log x} \frac{L^{k-2}}{(k-2)!} \e^{\ell + (k-2)\log\left(1-\frac{\ell}{L}\right)} \left(1 + \frac{(k-1 - L) H'(1) - H''(1)/2}{L} + \frac{(\ell-1) H'(1)}{L} + O\left(\frac{1}{L^2}\right)\right)
        \\&= \Pi_k(x) \frac{k-1}{L} \e^{\ell \frac{L-k+2}{L} - \frac{k\ell^2}{2L^2} + O\left(\frac{1}{L^2}\right)} \left(1 + \frac{(\ell-1) H'(1)}{L} + O\left(\frac{1}{L^2}\right)\right)
        \\&= \Pi_k(x) \left(1 + \frac{k-1-L + \ell(L-k+2) - \ell^2/2 + (\ell-1) H'(1)}{L} + O\left(\frac{1}{L^2}\right)\right),
    \end{align*}
    so that by formula (\ref{q-prime}), $$g_{k,x}(p) = \frac{k-1-L + \ell(L-k+2) - \ell^2/2 + (\ell-1) H'(1)}{L} + O\left(\frac{1}{L^2}\right)$$ as desired. 
\end{proof}

We have now all the ingredients to prove Theorem \ref{uniform-pik}.

\begin{proof}[Proof of Theorem \ref{uniform-pik}]
    We apply Theorem~\ref{y-le-q} to the normalised function $f$ defined by $$f(n) := \frac{x}{\Pi_k(x)} \Ind_{n\le x} \Ind_{\Omega(n) = k}$$ with $Q := \sqrt{yx^{1/3}} = x^{1/3-\eps/2}$ and $g = \Tilde{g}_{k,x} := g_{k,x} \Ind_{\Po},$ where $\Po$ is the set of primes. Then Lemma \ref{estim-pik} ensures $\Tilde{g}_{k,x} \ll 1,$ and the Sathe-Selberg estimates (Lemma \ref{sathe-selberg}) ensure $\frac{x}{\Pi_k(x)} \ll (\log x)^A$ for some $A$ when $k\le (2-\eps')\log\log x,$ so that $$\sum_{q\le Q} \big|\Tilde{g}_{k,x}(q)\big| \max_{(a,q) = 1} \left|F(x; a/q) - x\frac{\Tilde{g}_{k,x}(q)}{\phi(q)} \right| \ll \frac{x}{\Pi_k(x)} \sum_{p\le Q} \max_{(a,p) = 1} \left|\Pi_k(x; a/p) - \Pi_k(x)\frac{g_{k,x}(p)}{\phi(p)} \right| \ll \frac{x}{(\log x)^{C}}$$ for all $C$ and for $Q = x^{1/3-\eps/2}$ by Theorem \ref{main} and the bound (\ref{bv-e}) holds.

    Moreover, noting $\kappa := \frac{k}{\log\log x} \le 2-\eps',$ we have for $Q' \le Q,$
    \begin{equation}\label{step-int}
        \sum_{Q' < q \le 2Q'} \frac{|\Tilde{g}_{k,x}(q)|^2}{\phi(q)} = \sum_{Q' < p\le 2Q'} \frac{|g_{k,x}(p)|^2}{p-1} \gg \frac{1}{\log Q'} \left(\kappa\left(1 - \frac{\log Q'}{\log x}\right)^{\kappa-1} - 1 + O\left(\frac{1}{\log\log x}\right)\right)^2
    \end{equation}
    by Lemma \ref{estim-pik}. If $\kappa - 1 \gg 1,$ inequality (\ref{step-int}) implies for all $\delta > 0,$ $$\sum_{Q' < q \le 2Q'} \frac{|\Tilde{g}_{k,x}(q)|^2}{\phi(q)} \gg \frac{1}{\log Q'}$$ as long as $\left|\frac{\log Q'}{\log x} - \kappa^{-\frac{1}{\kappa-1}}\right| \ge \delta/20.$ In particular, summing over a larger interval $(1-\delta)\frac{\log Q}{\log x} \le \frac{\log Q'}{\log x} \le \frac{\log Q}{\log x},$ we get 
    $$\sum_{Q^{1-\delta} \le q \le Q} \frac{|\Tilde{g}_{k,x}(q)|^2}{\phi(q)} \gg 1$$ and the bound (\ref{goal-g}) holds with $\eta(Q) = \frac{1}{\log Q}$ when $\kappa-1 \gg 1.$

    Otherwise, estimating the right-hand side of (\ref{step-int}) around $\kappa=1,$ we get
    \begin{equation*}
        \sum_{Q' < q \le 2Q'} \frac{|\Tilde{g}_{k,x}(q)|^2}{\phi(q)} \gg \frac{1}{\log Q'} \left((\kappa-1)\left(1 + \log\left(1 - \frac{\log Q'}{\log x}\right)\right) + O\left((\kappa-1)^2\right) + O\left(\frac{1}{\log\log x}\right)\right)^2.
    \end{equation*}
    Since we always have $1+\log\left(1 - \frac{\log Q'}{\log x}\right) >0$ for $Q' \le Q$ as $\frac{\log Q}{\log x} < 1/3 < 1-1/\e,$ it follows $$\sum_{Q' < q \le 2Q'} \frac{|\Tilde{g}_{k,x}(q)|^2}{\phi(q)} \gg \frac{(\kappa-1)^2}{\log Q'}$$ whenever $\frac{C}{\log\log x} \le |\kappa-1| \le \delta'$ for some fixed $C, \delta'.$ In this case summing on $Q'$ gives us the lower bound (\ref{goal-g}) with $\eta(Q) = \frac{(\kappa-1)^2}{\log Q}.$

    It remains the case $\kappa-1 \ll \frac{1}{\log\log x},$ that is, $|k - \log\log x| \le C$ for some fixed $C.$ In that case, we can use the estimate (\ref{g-near}) to get 
    \begin{equation}\label{step-P}
        \sum_{Q' < q \le 2Q'} \frac{|\Tilde{g}_{k,x}(q)|^2}{\phi(q)} = \sum_{Q' < p\le 2Q'} \frac{|g_{k,x}(p)|^2}{p-1} \gg \frac{1}{\log Q'} \left(\frac{P_\kappa\left(-\log\left(1 - \frac{\log Q'}{\log x}\right)\right)}{\log\log x} + O\left(\frac{1}{(\log\log x)^2}\right)\right)^2
    \end{equation}
    where $P_\kappa(\ell) = -\ell^2/2 + a_1(\kappa)\ell + a_0(\kappa)$ with $a_1(\kappa), a_0(\kappa) = O(1)$ when $|k - \log\log x| \le C.$ In particular there are values $r_1(\kappa), r_2(\kappa)$ such that for all $\delta > 0,$ we have when $\left|\frac{\log Q'}{\log x} - r_1(\kappa)\right| \ge \delta/30$ and $\left|\frac{\log Q'}{\log x} - r_2(\kappa)\right| \ge \delta/30,$ $$\sum_{Q' < q \le 2Q'} \frac{|\Tilde{g}_{k,x}(q)|^2}{\phi(q)} \gg \frac{1}{\log Q'\cdot (\log\log x)^2}.$$ Since $\frac{\log Q'}{\log x}$ spans an interval of length $\delta\frac{\log Q}{\log x} \ge \delta/6$ when $Q^{1-\eps} \le Q'\le Q,$ we get $$\sum_{Q^{1-\eps} \le q \le Q} \frac{|\Tilde{g}_{k,x}(q)|^2}{\phi(q)} \gg \frac{1}{(\log\log x)^2}$$ and thus (\ref{goal-g}) holds with $\eta(Q) = \frac{1}{\log Q\cdot(\log\log x)^2}$ in this case.

    Thus, in all cases, the lower bound (\ref{goal-g}) holds for $\eta(Q) = \frac{1}{\log Q}\left(\frac{1 + |k-\log\log x|}{\log\log x}\right)^2,$ so that we get by Theorem~\ref{y-le-q} $$\frac{1}{x} \sum_{-y < n\le x} \left|\sum_{\substack{n< m\le n+y \\ 1 \le m \le x}} \frac{x}{\Pi_k(x)} f_k(m) \e(\alpha m)\right|^2 \gg y\left(\frac{1 + |k-\log\log x|}{\log\log x}\right)^2.$$
\end{proof}

\section*{Acknowledgements}

I would like to thank R\'egis de la Bret\`eche (Universit\'e Paris Cit\'e) and Sary Drappeau (Universit\'e Clermont Auvergne) for their guidance. 

\printbibliography

\end{document}